\documentclass[12pt,a4paper,reqno]{amsart}
\usepackage{amssymb,amsmath,amsthm,amstext,amsfonts,latexsym}
\usepackage[dvips]{graphicx}
\usepackage{color}

\makeatletter \@addtoreset{equation}{section} \makeatother

\renewcommand\thetable{\thesection.\@arabic\c@table}

\theoremstyle{plain}
\newtheorem{maintheorem}{Theorem}

\newtheorem{theorem}{Theorem }[section]
\newtheorem{proposition}[theorem]{Proposition}
\newtheorem{lemma}[theorem]{Lemma}

\theoremstyle{definition} \theoremstyle{remark}
\newtheorem{remark}[theorem]{Remark}

\DeclareMathAlphabet{\mathpzc}{OT1}{pzc}{m}{it}

\newcommand{\SL}{\text{SL}}

\renewcommand{\epsilon}{\varepsilon}

\begin{document}
\large

\title[A note on reversibility and Pell equations]{A note on reversibility and Pell equations}

\author[M. Bessa]{M\'ario Bessa}
\address{M\'ario Bessa\\ Departamento de Matem\'atica, Universidade da Beira Interior\\
Rua Marqu\^es d'\'Avila e Bolama, 6201-001 Covilh\~a, Portugal}
\email{bessa@ubi.pt}

\author[M. Carvalho]{Maria Carvalho}
\address{Maria Carvalho \\ Centro de Matem\'atica da Universidade do Porto,
Rua do Campo Alegre, 687, 4169-007 Porto, Portugal}
\email{mpcarval@fc.up.pt}

\author[A. Rodrigues]{Alexandre Rodrigues}
\address{Alexandre Rodrigues \\ Centro de Matem\'atica da Universidade do Porto \\
Rua do Campo Alegre, 687, 4169-007 Porto, Portugal}
\email{alexandre.rodrigues@fc.up.pt}

\date{\today}

\maketitle

\begin{abstract} We consider hyperbolic toral automorphisms which are reversible with respect to a linear area-preserving involution. We will prove that within this context reversibility is linked to a generalized Pell equation whose solutions we will analyze. Additionally, we will verify to what extent reversibility is a common feature and characterize the generic setting.
\end{abstract}

\medskip
{\small \noindent\emph{MSC2010:} primary 37D25, 37C80; secondary 37C05.\\
\emph{keywords:} Hyperbolic toral automorphisms; Reversibility; Generalized Pell equations.}

\section{Introduction}\label{se.intro}

Let $M$ be a compact, connected, smooth, Riemannian two-dimensional manifold without boundary and $\mu$ its normalized area. Denote by $\text{Diff}_{\mu}^{~1}(M)$ the set of all area-preserving $C^1$-diffeomorphisms of $M$ endowed with the $C^1$-topology. A diffeomorphism $f:M \rightarrow M$ is said to be Anosov if $M$ is a hyperbolic set for $f$, that is, if the tangent bundle of $M$ admits a splitting $E^s\oplus E^u$ such that there exist an adapted norm $\|\,\,\|$ and a constant $\sigma\in(0,1)$ satisfying
$$\|Df_x(v)\|\leq \sigma \quad \text{ and } \quad \|Df_x^{-1}(u)\|\leq \sigma$$
for every $x \in M$ and any unitary vectors $v\in E^s_x$ and $u\in E^u_x$. It is known, after \cite{F2, Man}, that among the $2$-dimensional manifolds only the torus (we will denote by $\mathbb{T}^2$) may support this type of diffeomorphisms. Moreover, each Anosov diffeomorphism on $\mathbb{T}^2$ is topologically conjugate to a linear model, that is, to a diffeomorphism induced on the torus by an element of the linear group $\SL(2,\mathbb{Z})$ of the $2\times2$ matrices with integer entries, determinant equal to $\pm 1$ and whose eigenvalues do not belong to the unit circle. In this note, we will consider linear Anosov diffeomorphisms in $\text{Diff}_{\mu}^{~1}(\mathbb{T}^2)$ which exhibit some symmetry. More precisely, let $R:\mathbb{T}^2\rightarrow \mathbb{T}^2$ be a diffeomorphism such that $R\circ R$ is the Identity map of $\mathbb{T}^2$ (such an $R$ is called an involution) and denote by $\text{Diff}^{~1}_{\mu, R}(\mathbb{T}^2)$ the subset of maps $f \in \text{Diff}_{\mu}^{~1}(\mathbb{T}^2)$, called $R$-reversible, such that $R$ conjugates $f$ and $f^{-1}$, that is,
$$R\circ f=f^{-1}\circ R.$$

Reversibility plays a fundamental role in physics and dynamical systems. In view of the many applications, the references \cite{LR} and \cite{RQ} present a thorough survey on reversible systems. Concerning this subject, important work was done in \cite{QC89} and \cite{BR}. In the former, the authors derive necessary conditions for local reversibility within mappings of the plane with a symmetric fixed point, expressing them through the quadratic and cubic coefficients of the Taylor expansion about the fixed point. This way, they establish an efficient negative criteria to show that a mapping is not reversible. On the other hand, on Section 2.2 of \cite{BR} we may read a detailed characterization of the kind of groups that are admissible as groups of reversible symmetries associated with unimodular matrices. Using algebraic techniques and results, the authors unveiled the possible structure of these groups, showing that it is completely resolvable when the matrices belong to $Gl(2,\mathbb{Z})$ or $PG(2,\mathbb{Z})$, as summarized in \cite[Theorem 3]{BR}.

Our work has followed a different, not purely algebraic approach and addressed other questions, thus complementing the information disclosed by the two previous references. First, we will present a simple method to construct a (non unique) Anosov diffeomorphism that anti-commutes with a given involution (Section~\ref{se.proof-main-teo}). Conversely, after concluding that an automorphism that reverses orientation is not reversible, we will establish a connection between the existence of reversible symmetries for a hyperbolic toral automorphism and the set of solutions of a generalized Pell equation (Section~\ref{se.Anosov}). Moreover, we will make clear in Section~\ref{se.generic} why generically, in the $C^1$ topology, the $r$-centralizer of a toral Anosov diffeomorphism is trivial, summoning a similar result obtained in \cite{BCW} within the conservative context for the centralizer. We believe that an extension of these results may be proved for reversing symmetries that are not involutions, and on higher dimensions. Yet, those are so far open problems.

\section{Overview}

Whereas it is often useful to check if a certain dynamical system is reversible under a given involution, we may also be interested in ascertaining if there is an involution under which a given dynamics is reversible, and to what extent this is a common feature. The latter query is hard to answer in general, and that is why we will confine our study to hyperbolic toral automorphisms, benefiting both from the linear structure and the low dimensional setting. In what follows we will address three questions. \\

\textbf{$Q_1$}. \emph{Given a linear involution $R\in\text{Diff}_{\mu}^{~1}(\mathbb{T}^2)$, is there a $R$-reversible linear (area-preserving) Anosov diffeomorphism $f$?} \\

The answer is obvious (and no) if $R = \pm Id$, the so called trivial cases. Concerning the other possible involutions, we will look for a linear Anosov diffeomorphism $f$ whose derivative at any point of $\mathbb{T}^2$ is a fixed linear map with matrix $L\in \SL(2,\mathbb{Z})$. To simplify our task, we will also assume that $det\,(L)=1$. Observe that, as $R$ is induced by a matrix $A \in \SL(2,\mathbb{Z})$ as well, if we lift the equality $R\circ f=f^{-1}\circ R$ by differentiating it at any point of $\mathbb{T}^2$, we obtain $A\circ L=L^{-1}\circ A.$ Analyzing the entries of these matrices we will answer positively to question $Q_1$.

\begin{proposition}\label{teo.reversing-linear-Anosov}
Let $\mathcal{L}$ be the set of linear Anosov diffeomorphisms on $\mathbb{T}^2$. If $R\in\text{Diff}_{\mu}^{~1}(\mathbb{T}^2)\setminus \{\pm \, Id\}$ is a linear involution, then
$$\text{Diff}^{~1}_{\mu, R}(\mathbb{T}^2) \cap \mathcal{L} \neq \emptyset.$$
\end{proposition}

The proof of this result will be presented on Section~\ref{se.proof-main-teo}. The argument is straightforward once the list of involutions in $\SL(2, Z)$ is determined. We observe that, in \cite[Lemma 4]{BR}, the authors work with involutions up to conjugacy to deduce a similar characterization.\\

Afterwards we will deal with the converse query.\\

\textbf{$Q_2$}. \emph{Given a linear Anosov diffeomorphism $f\in \text{Diff}_{\mu}^{~1}(\mathbb{T}^2)$, is there a linear involution $R\in\text{Diff}_{\mu}^{~1}(\mathbb{T}^2)$ such that $f$ is $R$-reversible?}\\

The answer depends essentially on whether $f$ preserves or reverses orientation. Indeed, as will be shown in Subsection~\ref{orientation-reversing}, if $f\in \text{Diff}_{\mu}^{~1}(\mathbb{T}^2)$ reverses orientation, meaning that $f$ is induced by a linear transformation $L\in \SL(2,\mathbb{Z})$ such that $det\,(L)=-1$, then $f$ is never $R$-reversible, for every linear involution $R\in \text{Diff}_{\mu}^{~1}(\mathbb{T}^2).$ We notice that this property may also be inferred from \cite[Lemma 6]{BR}.

Otherwise, if $f\in \text{Diff}_{\mu}^{~1}(\mathbb{T}^2)$ is an orientation-preserving diffeomorphism, then the reversibility of $f$ with respect to a linear involution $R$ determines a generalized Pell equation (see \cite{C, Mo} for this concept) whose solutions we will analyze on Section~\ref{se.Anosov}.

\begin{maintheorem}\label{teo.orientation-preserving}
Let $f\in \text{Diff}_{\mu}^{~1}(\mathbb{T}^2)$ be a linear Anosov diffeomorphism induced by a transformation
$$L=\begin{pmatrix} a & b \\ c & d \end{pmatrix}\,\,\in\,\, \SL(2,\mathbb{Z})$$
such that $det\,(L)=1$. Then:\\

\begin{itemize}
\item[(i)] $b$ divides $a-d$ if and only if $f$ is $R$-reversible, where $R$ is the projection on $\mathbb{T}^2$ of either
$$A=\begin{pmatrix} 1 & 0 \\ \frac{d-a}{b} & -1 \end{pmatrix} \quad \text{ or } \quad A=\begin{pmatrix} -1 & 0 \\ \frac{a-d}{b} & 1 \end{pmatrix}.$$

\item[(ii)] $c$ divides $a-d$ if and only if $f$ is $R$-reversible, where $R$ is the projection on $\mathbb{T}^2$ of either
$$A=\begin{pmatrix} 1 & \frac{d-a}{c} \\ 0 & -1 \end{pmatrix}\quad \text{ or } \quad A=\begin{pmatrix} -1 & \frac{a-d}{c} \\ 0 & 1 \end{pmatrix}.$$

\item[(iii)] Given $\alpha, \beta \in \mathbb{Z}\setminus\{0\}$ such that $1-\alpha^2 \neq 0$ and $\beta$ divides $1-\alpha^2$, consider the involution $R$ obtained by projecting on $\mathbb{T}^2$ the matrix
 $$A=\begin{pmatrix} \alpha & \beta \\ \frac{1-\alpha^2}{\beta} & -\alpha \end{pmatrix}\,\,\in\,\, \SL(2,\mathbb{Z}).$$
Then $f$ is $R$-reversible if and only if the generalized Pell equation $x^2-Dy^2=0$, where
$$D=(a+d)^2-4 \quad \text{ and } \quad N=4b^2,$$
has among its solutions
$$x=2b\alpha+(d-a)\beta \quad \text{ and } \quad y=\beta.$$
\end{itemize}
\end{maintheorem}

We remark that, if $f,g\in\text{Diff}^{~1}_{\mu, R}(\mathbb{T}^2)$, then $R\circ f^{-1}=f \circ R$ and $R\circ g^{-1}=g \circ R$ but
$$R\circ (f\circ g)=(f^{-1}\circ R)\circ g = (f^{-1}\circ g^{-1}) \circ R =(g \circ f)^{-1} \circ R$$
which means that the set $\text{Diff}^{~1}_{\mu,R}(\mathbb{T}^2)$, endowed with the composition of maps is, in general, not a group. Yet, $\text{Diff}^{~1}(\mathbb{T}^2)$, $\text{Diff}^{~1}_{\mu}(\mathbb{T}^2)$ and $\text{Diff}^{~1}_{\mu, R}(\mathbb{T}^2)$ are Baire spaces (see \cite{Devaney76} and Subsection~\ref{se.fixed points} for details), so it is natural to adopt another perspective concerning reversibility.\\

\textbf{$Q_3$}. \emph{What is the generic case?}\\

We will verify on Section~\ref{se.generic} that, for any diffeomorphism $f$ of a residual subset of $\text{Diff}^{~1}_{\mu}(\mathbb{T}^2)$, the equation $R\circ f=f^{-1}\circ R$ has only trivial solutions. Given a diffeomorphim $f:\mathbb{T}^2 \rightarrow \mathbb{T}^2$, the $r$-centralizer of $f$ is the set
$$\mathcal{Z}_r(f)=\{R \in \text{Diff}^{~1}(\mathbb{T}^2): R \circ f = f^{-1} \circ R\}.$$
$\mathcal{Z}_r(f)$ is said to be trivial if it is either empty or reduces to a set $\{R\circ f^n: n\in\mathbb{Z}\}$ for some $R \in \text{Diff}^{~1}(\mathbb{T}^2)$.

\begin{maintheorem}\label{teo.generic}
$C^1$-generically in $\text{Diff}^{~1}_{\mu}(M)$, the $r$-centralizer is trivial.
\end{maintheorem}

We note that, in spite of the fact that, on the torus $\mathbb{T}^2$, each Anosov diffeomorphism is conjugate to a hyperbolic toral automorphism, the conclusions we have drawn cannot be extended to all Anosov diffeomorphisms because the $R$-reversibility, for a fixed $R$, is not preserved by conjugacy. In fact, if $f\in\text{Diff}^{~1}_{\mu, R}(\mathbb{T}^2)$ is conjugate to $g\in\text{Diff}^{~1}_{\mu}(\mathbb{T}^2)$ through a homeomorphism $h$, then, although we have
$$(R\circ h)\circ g=f^{-1}(R\circ h),$$
the map $g$ may be not $R$-reversible. Nevertheless, the last equation indicates that
$$(h^{-1}\circ R\circ h)\circ g=g^{-1}\circ (h^{-1}\circ R \circ h),$$
so a diffeomorphism conjugate to a $R$-reversible linear Anosov diffeomorphism is reversible as well, although with respect to another involution, which is conjugate to $R$ but may be either non-linear or even non-differentiable.

\section{Linear involutions on $\mathbb{T}^2$}\label{se.involutions}

We start characterizing the linear involutions $R:\mathbb{T}^2\rightarrow \mathbb{T}^2$ of the torus, induced by matrices $A$ in $\SL(2,\mathbb{Z})$. Given this, we will be able to describe the set of fixed points of such involutions.

\subsection{Classification} After differentiating the equality $R^2=Id_{\mathbb{T}^2}$ at any point of $\mathbb{T}^2$, we obtain $A^2=Id_{\mathbb{R}^2}$. Comparing the entries of the matrices in this equality, we conclude that:

\begin{lemma}\label{prop.classification}
If $A\colon \mathbb{R}^2\rightarrow \mathbb{R}^2$ is linear involution of $\SL(2,\mathbb{Z})\setminus \{\pm Id\}$, then it belongs to the following list:\\
\begin{itemize}
\item $A=\begin{pmatrix}\pm1 & 0 \\ \gamma & \mp1 \end{pmatrix}$ \hspace{1cm} or its transpose, for some $\gamma\in\mathbb{Z}$.
\item $A=\begin{pmatrix} \alpha & \beta \\ \frac{1-\alpha^2}{\beta} & -\alpha \end{pmatrix}$ \hspace{0.8cm} for $\alpha,\beta\in\mathbb{Z}\backslash\{0\}$ such that $\beta$ divides $1-\alpha^2$.
\end{itemize}
\end{lemma}

\begin{proof}
Let $A$ be a matrix $\begin{pmatrix}\alpha & \beta \\ \gamma & \delta \end{pmatrix} \in \SL(2,\mathbb{Z})$ such that $A^2=Id$. This means that
$$\alpha \delta - \beta \gamma = \pm 1$$
and
\begin{equation*}
\left\{
\begin{array}{l}
\alpha^2+\beta\gamma=1\\
\beta(\alpha + \delta) = 0\\
\gamma(\alpha+\delta)=0\\
\gamma\beta+\delta^2=1
\end{array}
\right.
\end{equation*}
which implies that
\begin{equation*}
\left\{
\begin{array}{l}
\beta=0 \,\vee\, \alpha=-\delta\\
\gamma=0 \,\vee\, \alpha=-\delta.
\end{array}
\right.
\end{equation*}

\medskip

\noindent \textbf{1st case:} $\beta=0$

\medskip

One must have $\alpha=\pm 1$ and $\delta=\pm 1$. If $\alpha=\delta=1$ or $\alpha=\delta=-1$, we conclude that $\gamma=0$ and so $A=\pm Id$. Therefore, $-\alpha=\delta=1$ or $\alpha=-\delta=1$, and there are no restrictions on the value of $\gamma$. Hence
$$A = \begin{pmatrix}\pm1 & 0 \\ \gamma & \mp1 \end{pmatrix}\quad \text{ for } \, \gamma\in\mathbb{Z}.$$

\medskip

\noindent \textbf{2nd case:} $\gamma=0$

\medskip

Again $\alpha=\pm 1$ and $\delta=\pm 1$, and so either $-\alpha=\delta=1$ or $\alpha=-\delta=1$. Therefore $A = \begin{pmatrix}\pm 1 & \beta \\ 0 & \mp1 \end{pmatrix}$, $\beta\in\mathbb{Z}.$

\medskip

\noindent \textbf{3rd case:} $\beta \neq 0$ and $\gamma \neq 0$

\medskip

We must have $\alpha=-\delta$ and so $\alpha^2+\beta\gamma=1$ is equivalent to $\gamma=\frac{1-\alpha^2}{\beta}$. Moreover, as $A$ belongs to $\SL(2,\mathbb{Z})$, the entry $\beta$ must divide $1-\alpha^2$. Thus
$$A = \begin{pmatrix}\alpha & \beta \\ \frac{1-\alpha^2}{\beta} & -\alpha \end{pmatrix}.$$
\end{proof}

\subsection{Fixed points}\label{se.fixed points}

From the previous description of the matrices $A$ and by solving the equation $A(x,y)=(x,y)$ in $\mathbb{R}^2$, we deduce that the fixed point set of a linear non-trivial involution $R$ of the torus is a finite union of smooth closed curves obtained by projecting subspaces of $\mathbb{R}^2$ with dimension one.

Firstly, such an $R$ is induced by a matrix $A \in \SL(2,\mathbb{Z})$ as specified by Lemma~\ref{prop.classification}, so we get

\bigskip

\begin{center}
\small{\begin{tabular}{|c|c|}
	\hline
$A$ & Fixed point subspace of $A$ \\
	\hline\hline
$\begin{pmatrix} 1 & 0 \\ \gamma & - 1 \end{pmatrix}$ & $y=\frac{\gamma}{2}x$  \\
	\hline
$\begin{pmatrix} 1 & \gamma \\ 0 & -1 \end{pmatrix}$ & $y=0$  \\
	\hline
$\begin{pmatrix} -1 & 0 \\ \gamma & 1 \end{pmatrix}$ & $x=0$  \\
	\hline
$\begin{pmatrix} -1 & \gamma \\ 0 & 1 \end{pmatrix}$ & $x=\frac{\gamma}{2}y$  \\
	\hline
$\begin{pmatrix} \alpha & \beta \\ \frac{1-\alpha^2}{\beta} & -\alpha \end{pmatrix}, \,\,\beta \neq 0$ & $y=\frac{1-\alpha}{\beta}x $  \\
\hline
\end{tabular}}
\end{center}

\bigskip

\noindent Moreover, if $\pi:\mathbb{R}^2 \rightarrow \mathbb{T}^2=\mathbb{R}^2/\mathbb{Z}^2$ denotes the usual projection, then
$$A(x,y)=(x,y) \Rightarrow R(\pi(x,y))=\pi(x,y),$$
so the points of the projections on the torus of the subspaces of the previous table are fixed points by $R$. As the slopes of these lines are rational numbers, these projections are closed curves on $\mathbb{T}^2$.

Conversely, observe that $R(\pi(x,y))=\pi(x,y)$ if and only if $\pi(A(x,y))=\pi(x,y)$, and this implies that there exists $(n,m)\in \mathbb{Z}^2$ such that $A(x,y)=(x,y)+(n,m)$. So, for instance, if $A=\begin{pmatrix} 1 & 0 \\ \gamma & - 1 \end{pmatrix}$, then
$$A(x,y)=(x,y)+(n,m) \Leftrightarrow x \in \mathbb{R}, n= 0 \text{ and } y=\frac{\gamma}{2}\,x -\frac{m}{2}$$
and therefore, for the involution $R$ induced by such a matrix $A$, we have
$$Fix(R)=\pi\left(\left\{(x,\frac{\gamma}{2}\,x): x \in \mathbb{R}\right\}\right)\cup \pi\left(\left\{(x,\frac{\gamma}{2}\,x -\frac{1}{2}): x \in \mathbb{R}\right\}\right).$$
This is the union of two closed curves on $\mathbb{T}^2$. The other cases are analogous.

\section{Answer to question $Q_1$}\label{se.proof-main-teo}

Recall that, if a linear Anosov diffeomorphism $f$ is induced by a matrix $L(x,y)=(ax+by,cx+dy)$ of $\SL(2,\mathbb{Z})$, then the entries of $L$ must satisfy the conditions:\\

\begin{itemize}
\item[(IL)] (Integer lattice invariance) $\,\,a,b,c,d\in\mathbb{Z}$ and $ad-bc=\pm1$.
\item[(H1)] (Hyperbolicity) If $ad-bc=1$, then $\,\,(a+d)^2-4>0$.
\item[(H2)] (Hyperbolicity) If $ad-bc=-1$, then $\,\,(a+d)^2+4$ is not a perfect square.
\end{itemize}

\medskip

These requirements explain why a linear Anosov diffeomorphism is never $\pm Id$--reversible. Indeed, if $R=Id$ and $f$ is $R$-reversible, then $f^2=Id$; however, this equality does not hold among Anosov diffeomorphisms, whose periodic points are hyperbolic and so isolated. If $R=-Id$ and $f$ is induced by a matrix $L\in \SL(2,\mathbb{Z})$ and is $R$-reversible, then the equality $(-Id)\circ L=L^{-1}\circ (-Id)$ yields
\begin{equation*}
\left\{
\begin{array}{ll}
a=d,\,\, b=c=0, & \mbox{if $\,\,ad-bc=1$}\\
a=-d, & \mbox{if $\,\,ad-bc=-1$}\\
\end{array}
\right.
\end{equation*}
which contradicts one of the properties (H1) or (H2).\\

Going through the available matrices $A$, given by Lemma~\ref{prop.classification}, we will determine, for each $R$, an orientation-preserving $R$-reversible and linear Anosov diffeomorphism $f$.\\

\noindent \textbf{1.} $A=\begin{pmatrix} 1 & 0 \\ \gamma & - 1 \end{pmatrix}$

\medskip
	
We start noticing that the equality $A\circ L=L^{-1}\circ A$, with $det\,(L)=1$, is equivalent to $b\gamma=d-a$.

If $\gamma=0$, we may take $L=\begin{pmatrix} a & b \\ c & a \end{pmatrix}$ with integer entries such that $a^2-bc=1$ and $4a^2>4$ (so $b \neq 0$ and $c \neq 0$). For instance, $a=d=3$, $b=4$ and $c=2$.

If $\gamma \neq 0$, it must divide $d-a$ and $L$ has to be $\begin{pmatrix}a & \frac{d-a}{\gamma} \\ c & d \end{pmatrix}$, with integer entries such that $ad-bc=1$, $(a+d)^2>4$, $d-a \neq 0$ and $c\neq 0$. For example, $a=\gamma \in \mathbb{Z}\backslash \{0\}$, $b=1$, $c=2\gamma^2-1$ and $d=2\gamma$.

As the reversibility condition $A\circ L=L^{-1}\circ A$, with $det\,(L)=1$, is equivalent to $A^T\circ L^T=(L^T)^{-1}\circ A^T$, with $det\,(L^T)=1$, the case of the transpose matrix is equally solved.\\

\noindent \textbf{2.} $A=\begin{pmatrix}-1 & 0 \\ \gamma & 1 \end{pmatrix}$

\medskip

As in the previous case, the equality $A\circ L=L^{-1}\circ A$, with $det\,(L)=1$, is equivalent to $b\gamma=a-d$. So, if $\gamma=0$, we may take $L=\begin{pmatrix}3 & 4 \\ 2 & 3 \end{pmatrix}$. If $\gamma \neq 0$, we may choose, for instance, $a=\gamma \in \mathbb{Z}\backslash \{0\}$, $b=-1$, $c=1-2\gamma^2$ and $d=2\gamma$. Again, for $A^T= \begin{pmatrix}-1 & \beta \\ 0 & 1 \end{pmatrix}$, we may just pick the Anosov diffeomorphism induced by $L^T$.

\medskip

\noindent \textbf{3.} $A=\begin{pmatrix} \alpha & \beta \\ \frac{1-\alpha^2}{\beta} & -\alpha \end{pmatrix}$, with $\alpha, \beta, 1-\alpha^2 \neq 0$ and $\beta$ a divisor of $1-\alpha^2$

\medskip

The equality
$$\begin{pmatrix} \alpha & \beta \\ \frac{1-\alpha^2}{\beta} &  -\alpha \end{pmatrix} \begin{pmatrix}a & b \\ c & d \end{pmatrix}=\begin{pmatrix}d & -b \\ -c & a \end{pmatrix}\begin{pmatrix}\alpha & \beta \\ \frac{1-\alpha^2}{\beta} &  -\alpha \end{pmatrix}$$
is equivalent to the equation
$$\alpha a +\beta c= \alpha d - \frac{b}{\beta}(1-\alpha^2)$$
that is,
\begin{equation}\label{equation3}
\alpha\beta a +(1-\alpha^2) b +\beta^2 c -\alpha \beta d=0.
\end{equation}

\medskip
\noindent To easy our task, we may try to find a matrix satisfying $a=d$. Under this assumption, equation (\ref{equation3}) becomes
$$(1-\alpha^2) b +\beta^2 c=0.$$
As $c$ must also comply with the equality $a^2-bc=1$ and $b$ cannot be zero, we must have $c=\frac{a^2-1}{b}$. In addition, we know that $\beta$ divides $1-\alpha^2$ and that $\alpha^2\neq 1$, so $4\alpha^2 -4 > 0$. Therefore, a convenient choice is $a=d=\alpha$, $b=\pm\beta$ and $c=\frac{\alpha^2-1}{\pm\beta}$. Namely,
$$L=\begin{pmatrix} \alpha & \beta \\ \frac{\alpha^2-1}{\beta} & \alpha \end{pmatrix} \,\,\,\,\text{ or }\,\,\,\, L=\begin{pmatrix} \alpha & -\beta \\ \frac{1-\alpha^2}{\beta} & \alpha \end{pmatrix}.$$
This ends the proof of Theorem~\ref{teo.reversing-linear-Anosov}.

\section{Answer to question $Q_2$}\label{se.Anosov}

Let $f$ be a linear Anosov diffeomorphism, induced by a matrix $L=\begin{pmatrix} a & b \\ c & d \end{pmatrix}$ of $\SL(2,\mathbb{Z}).$

\subsection{Orientation-preserving case} Assume that $ad-bc=1$ and take a linear involution $R$, given by the projection on the torus of a matrix $A$ as described by Lemma~\ref{prop.classification}.\\

\noindent \textbf{Case 1}: $A=\begin{pmatrix} 1 & 0 \\ \gamma & -1 \end{pmatrix}$ or $A=\begin{pmatrix} -1 & 0 \\ \gamma & 1 \end{pmatrix}$.

\medskip

The reversibility equality is equivalent to $b\gamma = d-a$ or $b\gamma = a-d$. So there is such an involution $A$ if and only if $b$ divides $d-a$, in which case only one valid $\gamma$ exists (namely, $\gamma=\frac{d-a}{b}$ or $\gamma=\frac{a-d}{b}$, respectively).\\

\noindent \textbf{Case 2}: $A=\begin{pmatrix} 1 & \gamma \\ 0 & -1 \end{pmatrix}$ or $A=\begin{pmatrix} -1 & \gamma \\ 0 & 1 \end{pmatrix}$.

\medskip

Dually, the reversibility condition is equivalent to $c\gamma = d-a$ or $c\gamma = a-d$. So there is such an involution $A$ if and only if $c$ divides $d-a$, and then we get a unique value for $\gamma$.\\

\noindent \textbf{Case 3}: $A=\begin{pmatrix} \alpha & \beta \\ \frac{1-\alpha^2}{\beta} & -\alpha \end{pmatrix}$, where $\alpha, \beta, 1-\alpha^2 \neq 0$ and $\beta$ divides $1-\alpha^2$.

\medskip

The pairs $(\alpha,\beta)\in\mathbb{Z}^2$ for which $f$ is $R$-reversible are the integer solutions of the equation, in the variables $\alpha$ and $\beta$, given by
$$\alpha a + \beta c = \alpha d - \frac{b}{\beta}(1-\alpha^2)$$
that is,
$$b\alpha^2 + \alpha \beta (d-a) - \beta^2 c = b.$$
This quadratic form defines a conic whose kind depends uniquely on the sign of
$$\Delta=(d-a)^2 + 4bc = (a+d)^2 - 4$$
which we know to be always positive. So the conic is a hyperbola. After the change of variables
$$x=2b\alpha+(d-a)\beta \,\,\,\,\,\, \text{ and } \,\,\,\,\,\,y=\beta$$
the equation of the conic becomes
$$x^2-Dy^2=N$$
where $D=\Delta=(a+d)^2-4>0$ and $N=4b^2$. Thus the problem of finding the intersections of the conic with the integer lattice is linked to the solutions of this generalized Pell equation (and we need solutions with $y \neq 0$). According to \cite{B,C,Mo}, this sort of Pell equation has zero integer solutions or infinitely many, and there are several efficient algorithms to determine which one holds in each particular case. However, notice that, if they exist, the solutions we are interested in have also to fulfill the other requirements, namely $\alpha, \beta \neq 0$ and $\beta$ divides $1-\alpha^2$.

\subsection{Example 1} It is time to test the previous information in a few examples.\\

\begin{center}
\small{\begin{tabular}{|c|c|c|c|c|c|}
	\hline
 \textbf{\emph{Anosov}} & & & \textbf{\emph{Involutions}} & &  \\
\hline\hline
& $\begin{pmatrix} 1 & 0 \\ \gamma & -1 \end{pmatrix}$ & $\begin{pmatrix} 1 & \gamma \\ 0 & -1 \end{pmatrix}$ & $\begin{pmatrix} -1 & 0 \\ \gamma & 1 \end{pmatrix}$ &  $\begin{pmatrix} -1 & \gamma \\ 0 & 1 \end{pmatrix}$ & $\begin{pmatrix} \alpha & \beta \\ \frac{1-\alpha^2}{\beta} & -\alpha \end{pmatrix}$ \\
\hline\hline
$\begin{pmatrix} 2 & 1 \\ 3 & 2 \end{pmatrix}$ & $\gamma=0$ & $\gamma=0$ & $\gamma=0$ & $\gamma=0$ & For instance, $\,\begin{pmatrix} 2 & 1 \\ -3 & -2 \end{pmatrix}$ \\
\hline
$\begin{pmatrix} 2 & 1 \\ 1 & 1 \end{pmatrix}$ & $\gamma=-1$ & $\gamma=-1$ & $\gamma=1$ & $\gamma=1$ & For instance, $\,\begin{pmatrix} 5 & 3 \\ -8 & -5 \end{pmatrix}$ \\
\hline
$ \begin{pmatrix} 4 & 9 \\ 7 & 16 \end{pmatrix}$ & $-$ & $-$ & $-$ & $-$ & $-$ \\
	\hline
\end{tabular}}
\end{center}

\medskip

Let us follow in detail the entries of this table. When $L=\begin{pmatrix} 2 & 1 \\ 3 & 2 \end{pmatrix}$, the generalized Pell equation is $x^2-12y^2=4$ and there are infini\-tely many matrices $A$ of the third kind which correspond to linear involutions $R$ such that $f$ is $R$-reversible. Similarly, for $L=\begin{pmatrix} 2 & 1 \\ 1 & 1 \end{pmatrix}$ the generalized Pell equation is $x^2-5y^2=4$ and there are infinitely many involutions of type $3$. The third example in this table has no linear involutions, although its Pell equation $x^2-396y^2=324$ has infinitely many solutions.

\begin{remark}
Notice that, if $R$ is an involution such that $R \circ f= f^{-1} \circ R$, then, for each $n\in \mathbb{Z}$, the diffeomorphism $R \circ f^n$ is also an involution, since
$$(R \circ f^n)^2 = (R \circ f^{n}) \circ (f^{-n}\circ R) = Id$$
and $f$ is $R \circ f^n$-reversible due to the equality
$$(R \circ f^n)\circ f= (R \circ f) \circ f^n=(f^{-1}\circ R) \circ f^n = f^{-1}\circ (R \circ f^n).$$
Therefore, once such an involution $R$ is found for an Anosov diffeomorphism $f$, then we have infinitely many involutions with respect to which $f$ is reversible. This is so because, as no non-trivial power of an Anosov diffeomorphism is equal to the Identity, we have $R\circ f^k \neq R\circ f^m$ for every $k\neq m \in \mathbb{Z}$.
\end{remark}

\subsection{Orientation-reversing case}\label{orientation-reversing}

Consider now a linear Anosov diffeomorphism $f$, induced by a matrix $L=\begin{pmatrix}a & b \\ c & d \end{pmatrix}$ of $\SL(2,\mathbb{Z})$ such that $ad-bc=-1$. The previous analysis extends to this setting with similar conclusions.\\

\noindent \textbf{Cases 1,\,2}: There is no valid $A$, since the reversibility condition $A\circ L=L^{-1}\circ A$ demands that $b$ and $a+d$ are zero, and these values are forbidden in the hyperbolic context.

\medskip

\noindent \textbf{Case 3}: $A=\begin{pmatrix} \alpha & \beta \\ \frac{1-\alpha^2}{\beta} & -\alpha \end{pmatrix}$, where $\alpha, \beta, 1-\alpha^2 \neq 0$ and $\beta$ divides $1-\alpha^2$.

\medskip

The pairs $(\alpha,\beta)\in\mathbb{Z}^2$ for which $f$ is $R$-reversible form the set of integer solutions of the equations, in the variables $\alpha$ and $\beta$, given by
\begin{eqnarray*}
\alpha b + \beta d &=& 0 \\
\alpha c - \frac{a}{\beta}(1-\alpha^2)&=&0 \\
\alpha a + \beta c &=& -\alpha d + \frac{b}{\beta}(1-\alpha^2).
\end{eqnarray*}
The last equality describes a (possibly degenerate) conic
$$b\alpha^2 + \alpha \beta (a+d)+ \beta^2 c = b$$
whose type is determined by the sign of
$$\Delta=(a+d)^2 -4bc = (a-d)^2 - 4.$$
For instance, $\Delta <0$ for $L=\begin{pmatrix} 2 & 3 \\ 1 & 1 \end{pmatrix}$; $\Delta = 0$ when $L=\begin{pmatrix}3 & 4 \\ 1 & 1 \end{pmatrix}$; and $\Delta >0$ if $L=\begin{pmatrix} 4 & 5 \\ 1 & 1 \end{pmatrix}$.

Once again, the problem of finding the points of this conic in the integer lattice is linked to the existence of solutions of the generalized Pell equation $x^2-Dy^2=N$, where $D=\Delta=(a-d)^2-4$ and $N=4b^2$, satisfying
\begin{eqnarray*}
x&=&2b\alpha+(a+d)\beta \\
y&=&\beta
\end{eqnarray*}
under the other constraints, namely
$$\left\{ \begin{array}{c}
\alpha, \,\beta,\, 1-\alpha^2 \neq 0 \\
\beta\, \text{ divides } \,1-\alpha^2 \\
\alpha b + \beta d = 0 \\
\alpha c - \frac{a}{\beta}(1-\alpha^2)=0.
\end{array}
\right.$$

\medskip

\subsection{Example 2} Let us examine some orientation-reversing examples.\\
\begin{center}
\small{\begin{tabular}{|c|c|c|c|c|c|}
	\hline
 \textbf{\emph{Anosov}} & \textbf{$\Delta$}  &\emph{\textbf{Pell Equation}} & \textbf{\emph{Number of solutions}} & \textbf{\emph{Conic}} &\textbf{\emph{Involutions}}  \\
\hline\hline
$\begin{pmatrix} 2 & 3 \\ 1 & 1 \end{pmatrix}$ & $-3$ & $x^2+3y^2=36$ & $6$ & Ellipse & $-$  \\
\hline
$\begin{pmatrix} 3 & 4 \\ 1 & 1 \end{pmatrix}$ & $0$ & $x^2=64$ & $\infty$ & Two vertical lines & $-$ \\
\hline
$\begin{pmatrix} 4 & 5 \\ 1 & 1 \end{pmatrix}$ & $5$ & $x^2-5y^2=100$ & $\infty$ & Hyperbola & $-$ \\
	\hline
\end{tabular}}
\end{center}

\bigskip

This short sample of examples seems nonetheless to hint that, no matter the amount of solutions the corresponding generalized Pell equation possesses, a linear orientation-reversing Anosov diffeomeorphism is not reversible with respect to linear involutions of the third sort (that is, projections of matrices
$$A=\begin{pmatrix} \alpha & \beta \\ \frac{1-\alpha^2}{\beta} & -\alpha \end{pmatrix} \,\,\in\, \SL(2,\mathbb{Z})$$
with $\alpha, \beta \neq 0$). And this is precisely the case, as we will now verify.\\

Consider one such a matrix $A$ and let us go back to the three conditions arising from reversibility in this setting:\\
$$\left\{ \begin{array}{c}
\alpha b + \beta d = 0 \\
\alpha \beta c - a(1-\alpha^2)=0 \\
b\alpha^2 + \alpha \beta (a+d)+ \beta^2 c = b.
\end{array}
\right.$$

\medskip

\noindent Replacing on the third equality $\alpha b$ by $-\beta d$, we get
$$\alpha \beta a +\beta^2 c = b.$$
Then, multiplying this equation by $\alpha$, which is nonzero, and turning $\alpha \beta c$ into $a(1-\alpha^2)$, we arrive at
$$\beta a = \alpha b.$$
This, joined to $\alpha b =- \beta d$, yields
$$\beta(a+d)=0.$$
As $\beta \neq 0$, we must have $a+d=0$, a value banned by the hyperbolicity of $f$.

\section{Answer to question $Q_3$}\label{se.generic}

Given an area-preserving diffeomorphim $f$, if $f^2=Id$, then $f^n$ belongs to $\mathcal{Z}_r(f)$ for all $n \in \mathbb{Z}$; and conversely. However, a generic $f \in \text{Diff}^{~1}_{\mu}(M)$ does not satisfy the equality $f^n=Id$, for any integer $n\neq 0$. Indeed, the Kupka-Smale Theorem for area-preserving diffeomorphisms \cite{Devaney76} asserts that, given $k\in \mathbb{N}$, generically in the $C^1$ topology, the periodic orbits of period less or equal to $k$ are isolated.

The existence of a hyperbolic toral diffeomorphism, which is structurally stable, without reversible symmetries shows that there are non-empty open subsets of $\text{Diff}^{~1}_{\mu}(M)$ where reversibility is absent. Moreover, if $R\neq S$ are in $\mathcal{Z}_r(f)$, then $R \circ S$ belongs to the centralizer of $f$, due to the equalities
$$(R \circ S) \circ f = R \circ (S \circ f) = R \circ (f^{-1} \circ S) = f \circ (R \circ S).$$
Now, according to \cite{BCW}, for a $C^1$-generic $f \in \text{Diff}^{~1}_{\mu}(M)$, the centralizer of $f$ is trivial, meaning that it reduces to the powers of $f$. Therefore, there must exist $n \in \mathbb{Z}$ such that $S=R\circ f^n$. Thus, if $\mathcal{Z}_r(f)\neq \emptyset$, then its elements are obtained from the composition of one of them with the powers of $f$. And so, $C^1$-generically, $\mathcal{Z}_r(f)$ is trivial.\\

\small
\section*{Acknowledgements}
 M.~Bessa and A.~Rodrigues thank Ant\'onio Machiavelo for various useful discussions. M.~ Bessa was partially supported by National Funds through FCT (Funda\c{c}\~{a}o para a Ci\^{e}ncia e a Tecnologia) project PEst-OE/MAT/UI0212/\-2011. M.~Carvalho and A.~Rodrigues were funded by CMUP (UID/MAT/00144/2013), which is financially supported by FCT (Portugal) with national (MEC) and European structural funds through the programs FEDER, under the partnership agreement PT2020. A.~Rodrigues has also benefited from the FCT grant SFRH/BPD/84709/2012.\\

\end{document}